\documentclass[11pt
]{article}
\usepackage{url}
\usepackage{graphics}
\usepackage{graphicx}
\usepackage{color}
\usepackage{amsmath}
\usepackage{amsfonts}
\usepackage{amsthm,amssymb}
\usepackage{tikz,color}

\def\vector#1{{\mathbf {#1}}}
\def\star#1{{\rm star}(#1)}
\def\type#1{{\rm type}(#1)}
\def\lb{{\rm lb}}
\def\smtyp{{\hbox{\rm\tiny
    Types}(m)}}

\def\numcr{{\times}}
\newcommand{\Cr}{{\hbox{\rm cr}}}
\newcommand{\Crt}{{{\nu}_2}} 

\newcommand{\dd}{{\cal D}}

\newcommand\floor[1]{{\lfloor{#1}\rfloor}}

\newcommand{\smallfloor}[1]{{\bigl\lfloor{#1}\bigr\rfloor}}

\newcommand{\Diag}{\mbox{Diag}}

\newcommand{\trace}{\mathrm{trace}}

\newcommand{\beq}{\begin{equation}}
\newcommand{\eeq}{\end{equation}}
\newcommand{\beann}{\begin{eqnarray*}}
\newcommand{\eeann}{\end{eqnarray*}}
\newcommand{\bc}{\begin{center}}
\newcommand{\ec}{\end{center}}

\newtheorem{theorem}{Theorem}
\newtheorem{lemma}[theorem]{Lemma}
\newtheorem{claim}[theorem]{Claim}

\newtheorem{proposition}[theorem]{Proposition}

\title{Improved lower bounds for the 2-page crossing numbers \\ of $K_{m,n}$
  and $K_{n}$ via semidefinite programming}
\author{E.~de Klerk\thanks{Department of Econometrics and OR, Tilburg University, The Netherlands.}
  \and
D.V.~Pasechnik\thanks{School of Physical \& Mathematical Sciences, Nanyang Technological University,
             Singapore.}
}
\begin{document}
\maketitle
\vspace{-1cm}
\begin{abstract}
It has been long conjectured that the
crossing numbers of the complete bipartite graph $K_{m,n}$ and of the
complete graph $K_n$ equal
$Z(m,n):=\smallfloor{\frac{n}{2}}
\smallfloor{\frac{n-1}{2}}
\smallfloor{\frac{m}{2}}\smallfloor{\frac{m-1}{2}}$ and
$Z(n):=\frac{1}{4}\smallfloor{\frac{n}{2}}
\smallfloor{\frac{n-1}{2}}
\smallfloor{\frac{n-2}{2}}\smallfloor{\frac{n-3}{2}}$, respectively.
In a $2$-{\em page drawing} of a graph, the vertices are drawn on a
straight line (the {\em spine}), and each edge is contained in one of
the half-planes of the spine.
The $2$-{\em page crossing number} $\Crt(G)$ of a
graph $G$ is the minimum number of crossings in a $2$-page drawing of
$G$. 
Somewhat
surprisingly, there are $2$-page drawings of $K_{m,n}$ (respectively,
$K_n$) with exactly $Z(m,n)$ (respectively, $Z(n)$)
crossings, thus yielding the conjectures
(I) $\Crt(K_{m,n}) \stackrel{?}{=} Z(m,n)$ and (II)
$\Crt(K_n) \stackrel{?}{=} Z(n)$.
It is known that (I) holds
for $\min\{m,n\} \le 6$, and that (II) holds for $n \le 14$.
In this paper we prove that (I) holds asymptotically (that is,
$\lim_{n\to\infty} \Crt(K_{m,n})/Z(m,n) =1$) for $m=7$ and
$8$. We also
prove (II)
for $15 \le n \le 18$ and $n \in \{20,24\}$, and establish the asymptotic
estimate
\[
\lim_{n\to\infty}
\Crt(K_{n})/Z(n) \ge 0.9253.
\]
The previous best-known lower bound
involved the constant $0.8594$.

%

\end{abstract}
{\bf Keywords:} $2$-page crossing number, book crossing number,
semidefinite pro\-gram\-ming, maximum cut, Goemans-Williamson max-cut bound

{\bf AMS Subject Classification:} 90C22, 90C25, 05C10, 05C62, 57M15, 68R10

\section{Introduction}

 We recall that the {\em crossing number} $\Cr(G)$ of a graph $G$ is the minimum number of
pairwise intersections of edges (at a point other than a vertex)  in a drawing
of $G$ in the plane. Besides their natural interest in topological
graph theory, crossing number problems are of interest because of
their applications, most notably in VLSI design~\cite{leightonvlsi}.


Also motivated by applications to VLSI design, Chung, Leighton and
Rosenberg~\cite{leighton} studied embeddings of graphs in {\em books}:
the vertices are placed along a line (the {\em spine}) and the edges
are placed in the {\em pages} of the book.
In a {\em book drawing} (equivalently, \emph{$k$-page drawing}, if the book
has $k$ pages), crossings among edges are allowed. The
$k$-{\em page crossing number} $\nu_k(G)$ of a graph $G$ is the
minimum number of crossings of edges in a $k$-page drawing of $G$.

Clearly, a graph $G$ has $\nu_1(G) = 0$ if and only if it is
outerplanar.
Closely related to $1$-page drawings are {\em circular
  drawings}, in which the vertices are placed on a circle and
all edges are drawn in its interior.
It is easy to see the one-to-one
correspondence between $1$-page drawings and circular drawings.

In a
similar vein, $2$-page drawings can be alternatively modelled by
drawing the vertices of the graph on a circle, and imposing the
condition that every edge lies either in the interior or in
the exterior of the circle (see Figure~\ref{fig:2pagemodels}).  In this paper we shall often use this
equivalent {\em circular model} for
$2$-page drawings, as well as the usual {\em spine model}.
It is known that the family of graphs $G$ with $\nu_2(G) = 0$ is precisely the 
family of subgraphs of Hamiltonian planar graphs \cite{Bernhart-Kainen}. As a consequence,
there exist planar graphs $G$ with $\nu_2(G) > 0$, in contrast to the case of the normal crossing number.
In fact, it was shown that all planar graphs may be embedded without crossings in 4-page books, and that 
four pages are necessary \cite{Yannanakis}.

\begin{figure}[ht]
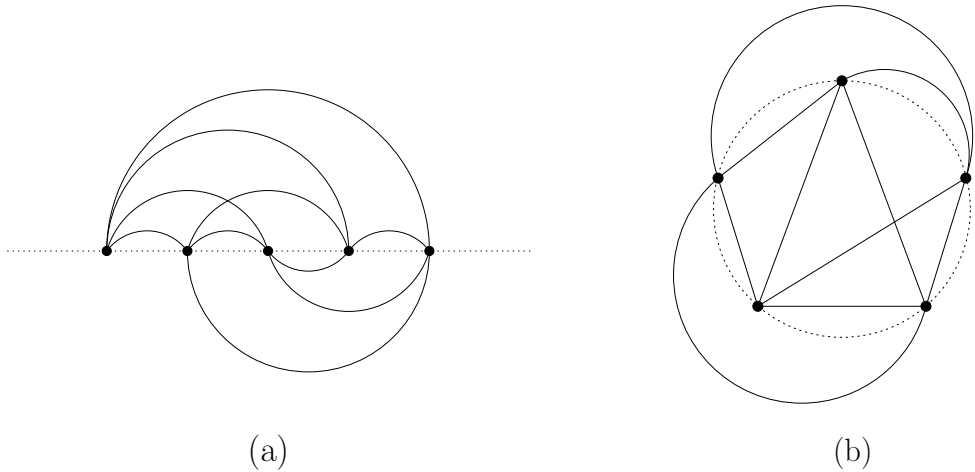

\label{fig:2pagemodels}
\begin{minipage}[b]{0.5\linewidth}
\centering
\resizebox{7cm}{!}{\input{k5-01.pspdftex}}
\end{minipage}
\hspace{-0.5cm}
\begin{minipage}[b]{0.5\linewidth}
\centering
\resizebox{4cm}{!}{\input{k5-02.pspdftex}}
\end{minipage}
\caption{A $2$-page drawing of $K_5$: (a) in the spine model;
  and (b) in the circular model.}
\end{figure}

Masuda et al.~\cite{masuda0,masuda} proved that the decision problems
for $\nu_1$ and $\nu_2$ are NP-complete.
Shahrokhi et al.~\cite{sssv} gave an approximation algorithm for
$\nu_k(G)$, as well as applications to the rectilinear crossing
number. A more recent, additional motivation for studying $k$-page
crossing numbers comes from Quantum Dot Cellular Automata~\cite{tl}.

Several interesting algorithms and heuristics have been proposed for producing 1- and
2-page drawings (see for instance \cite{cimi,cimi2,he2,he3,he4,he5}). As with the
usual crossing number, the exact computation of $\nu_k(G)$ (for any
integer $k$) is a very challenging problem, even for restricted
families of graphs.  In this direction, Fulek, He, S\'ykora, and
Vrt'o~\cite{fulek},
He, S\v{a}l\v{a}gean, and
M\"akinen~\cite{he1}, and Riskin~\cite{riskin} have computed the exact $1$-page and $2$-page
crossing numbers of several interesting families of graphs.


\subsection{Drawings of $K_{m,n}$ and $K_n$}

Tur\'an asked in the 1940's: what is the crossing number of the
complete bipartite graph $K_{m,n}$?
There is a natural drawing of $K_{m,n}$ with exactly $Z(m,n):=\smallfloor{\frac{n}{2}}
\smallfloor{\frac{n-1}{2}}
\smallfloor{\frac{m}{2}}\smallfloor{\frac{m-1}{2}}$ crossings (see Figure~\ref{fig:kmn}), and so
$\Cr(K_{m,n}) \le Z(m,n)$.

\begin{figure}[h!]
\begin{center}
\includegraphics[width=6cm]{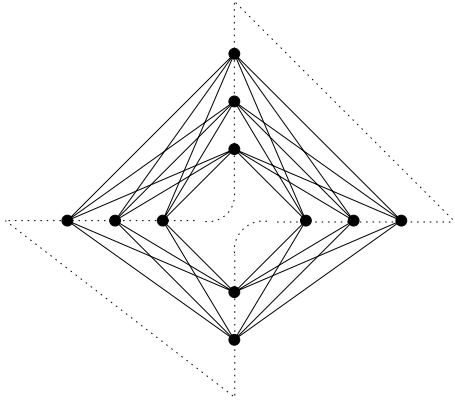}
\caption{A drawing of $K_{5,6}$ with  $Z(5,6)=24$ crossings. By
  performing a homeomorphism from the plane to itself that takes the
  dotted curve to a straight line, the result is a $2$-page drawing
  of $K_{5,6}$ with the same number of crossings.}
\label{fig:kmn}
\end{center}
\end{figure}

Perhaps the foremost open crossing number problem is {\em
 Zarankiewicz's Conjecture}, dating back to the early 1950's~\cite{zaran}:
\begin{equation}\label{eq:zarank}
\Cr(K_{m,n}) \stackrel{?}{=}
Z(m,n).
\end{equation}

This conjecture has been verified only for $\min\{m,n\} \le 6$~\cite{kleitman}, and
for the special cases $(m,n) \in \{(7,7), (7,8), (7,9), (7,10), (8,8), (8,9),(8,10)\}$~\cite{woodall}.

On a parallel front, there are drawings of the complete graph $K_n$ with exactly $Z(n):= \frac{1}{4}\smallfloor{\frac{n}{2}}
\smallfloor{\frac{n-1}{2}}
\smallfloor{\frac{n-2}{2}}\smallfloor{\frac{n-3}{2}}$ crossings (for
every $n$), and
so $\Cr(K_n) \le Z(n)$.
These drawings inspired the still open, long-standing Harary-Hill Conjecture~\cite{hararyhill}:
\begin{equation}\label{eq:conj1}
\Cr(K_n) \stackrel{?}{=} Z(n).
\end{equation}
This conjecture has been
verified for $n \le 12$~\cite{panrichter}.

For a detailed account on the history of (\ref{eq:zarank}) and (\ref{eq:conj1}),
we refer the reader to the lively survey by Beineke
and Wilson~\cite{bw}.

\subsection{$2$-page drawings of $K_{m,n}$ and $K_n$}

The drawing in Figure~\ref{fig:kmn} is easily generalized to yield a
drawing of $K_{m,n}$ with $Z(m,n)$ crossings. As mentioned in the
caption of this figure, such a drawing is easily transformed into a
$2$-page drawing of $K_{m,n}$ with the same number of crossings.
Thus, there exist $2$-page drawings of $K_{m,n}$ with
$Z(m,n)$ crossings.

On the other hand, it is somewhat surprising that
there exist $2$-page drawings of $K_n$
 with exactly $Z(n)$ crossings, for every positive integer $n$ (\cite{guy2}; see
also~\cite{Harborth}).

These observations imply that $\Crt(K_{m,n}) \le Z(m,n)$ and
$\Crt(K_n) \le Z(n)$. Since obviously $\Cr(G) \le \Crt(G)$ for every
graph $G$, (\ref{eq:zarank}) and (\ref{eq:conj1})
immediately imply the following conjectures:
\begin{equation}\label{eq:conj2pkmn}
\Crt(K_{m,n}) \stackrel{?}{=} Z(m,n).
\end{equation}
\begin{equation}\label{eq:conj2p}
\Crt(K_n) \stackrel{?}{=} Z(n).
\end{equation}

Even though (\ref{eq:conj2pkmn}) and (\ref{eq:conj2p}) are (at least in
principle) weaker than the corresponding (\ref{eq:zarank}) and
(\ref{eq:conj1}), and
even though the $2$-page crossing number problem can be naturally formulated in purely combinatorial
terms, our current knowledge (prior to this paper) on
(\ref{eq:conj2pkmn}) and (\ref{eq:conj2p})
is not substantially better than our knowledge on
(\ref{eq:zarank}) and (\ref{eq:conj1}). Indeed, the only step ahead is the proof by Buchheim
and Zheng~\cite{Buchheim-Zheng} that $\Crt(K_{13}) = Z(13)$ (from which a routine counting
argument yields that $\Crt(K_{14}) = Z(14)$). The best general lower bounds known
for $\Crt(K_{m,n})$ and $\Crt(K_n)$ are the same as those known for
$\Cr(K_{m,n})$ and $\Cr(K_n)$, and the same
is true for the asymptotic ratio $\lim_{n\to\infty} \Crt(K_n)/Z(n)$, whose
best current estimate is exactly the same as the asymptotic ratio
$\lim_{n\to\infty} \Cr(K_n)/Z(n)$, namely $0.859$~\cite{DeKPasSch}.

\subsection{Main results}
Our main results in this paper
offer a substantial improvement on our knowledge of $\Crt(K_{m,n})$
and $\Crt(K_n)$ over our knowledge of $\Cr(K_{m,n})$ and $\Cr(K_n)$.

\renewcommand{\theequation}{{\rm A}}

\begin{theorem}\label{thm:main2}
The $2$-page Harary-Hill Conjecture holds for all $m \le 18$ and for
$m= 20$ and $24$:
\begin{equation}\label{eq:main1}
\Crt(K_{m}) = Z(m) \ \hbox{ for all }\ m \le 18 \ \hbox{ and for }\ m \in \{20,24\}.
\end{equation}
\addtocounter{equation}{-1}
 \renewcommand{\theequation}{{\rm B}}
\noindent Moreover, the asymptotic ratio between the $2$-page crossing number of
$K_n$ and its conjectured value satisfies:
\begin{equation}\label{eq:main2}
\lim_{n\to\infty} \frac{\Crt(K_{n})}{Z(n)} \ge 0.9253.
\end{equation}
\end{theorem}

\begin{theorem}
The $2$-page Zarankiewicz's Conjecture holds in the asymptotically
relevant term for $m=7$
and $8$. That
is:
\label{thm:main1}
\begin{align*}
\nonumber
\Crt(K_{7,n}) &= (9/{4})n^2 + O(n) = Z(7,n) + O(n),\ \hbox{\rm and} \\
\Crt(K_{8,n}) &= 3n^2 + O(n) = Z(8,n) + O(n).
\end{align*}
Therefore,
\begin{equation*}
\lim_{n\to\infty} \frac{\Cr(K_{7,n})}{Z(7,n)} = 1 \hbox{\hglue 0.5 cm
  \rm and \hglue 0.5 cm} \lim_{n\to\infty} \frac{\Cr(K_{8,n})}{Z(8,n)} = 1.
\end{equation*}
\end{theorem}

\addtocounter{equation}{-1}

\renewcommand{\theequation}{\arabic{equation}}

%


\subsubsection*{Outline of this paper}
The rest of this paper is structured as follows. In
Section~\ref{sec:maxcutformulation}, we review the
reformulation (first
unveiled by Buchheim and Zheng~\cite{Buchheim-Zheng}) in which the
problem of calculating $\Crt(K_n)$ is shown to be equivalent to a
maximum cut problem on an associated graph $G_n$.  In Section 3 we
invoke a result by Goemans and Williamson that provides an upper bound
on the size of the maximum cut of a graph; this bound may be computed
via semidefinite programming. Using these ingredients, in Section~\ref{sec:numerical} we
present the numerical computations that establish
Theorem~\ref{thm:main2}.  In Section~\ref{sec:kmnone} we formulate a
quadratic program whose solution yields a lower bound on
$\Crt(K_{m,n})$. In Section~\ref{sec:kmntwo} we analyze the
semidefinite programming relaxation of this quadratic program, and in
Section~\ref{sec:numericalzar} we give the numerical computations that
prove Theorem~\ref{thm:main1}. In
Section~\ref{sec:concludingremarks} we present some concluding
remarks.

\section{Formulating $\Crt(K_{n})$ as a maximum cut
  problem}\label{sec:maxcutformulation}

Buchheim and Zheng \cite{Buchheim-Zheng} unveiled a natural reformulation of
the fixed linear crossing number problem (FLCNP) as a maximum cut
problem. Their results imply, in particular, that $\Crt(K_{n})$ can be obtained by computing the maximum cut size
in a certain graph $G_n=(V_n,E_n)$, with $V_n$ and $E_n$ defined as
follows.

Consider a Hamiltonian cycle with vertices $v_1, v_2, \ldots, v_n$. Let $V_n$
be the set of {\em chords} of the cycle, that is, the edges $v_iv_j$ with
$v_i$ and $v_j$ at cyclic distance at least $2$. Now to define $E_n$, let two chords $v_iv_j$ and
$v_kv_\ell$ be adjacent if they intersect. This construction is
illustrated in Figure \ref{fig:G5} for $n = 5$.

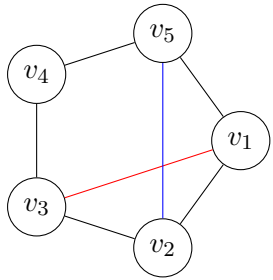
\begin{figure}[h!]
\begin{center}
\begin{tikzpicture}[scale=1.5]
\tikzstyle{every node}=[draw,shape=circle];
\path (0:1cm) node (v1) {$v_1$};
\path (72:1cm) node (v2) {$v_5$};
\path (2*72:1cm) node (v3) {$v_4$};
\path (3*72:1cm) node (v4) {$v_3$};
\path (4*72:1cm) node (v5) {$v_2$};
\draw (v1) -- (v2)
(v2) -- (v3)
(v4) -- (v3)
(v5) -- (v4)
(v1) -- (v5);
{\color{red}\draw (v1) -- (v4);}
{\color{blue}\draw (v2) -- (v5);}
\end{tikzpicture}
\caption{\label{fig:G5} The chords $v_1v_3$ and $v_2v_5$ form adjacent vertices in the graph $G_5$.}
\end{center}
\end{figure}

Thus $|V_n| = {n \choose 2} - n$, and it is easy to check that $|E_n| = {n \choose 4}$.
The automorphism group of $G_n$ is isomorphic to the dihedral group $D_n$, and there are $d-1$ orbits
of vertices, where $d = \floor{n/2}$. The equivalency classes of vertices (i.e.\ vertices belonging to the same orbit)
may be described as follows: since vertices correspond to chords in $P_n$, the chords that connect vertices of $P_n$
at the same cyclic distance belong to the same equivalency class.
The vertices corresponding to chords with cyclic distance $i$ have valency  $i(i-1) + 2(i-1)(d-i)$, as is easy to check.

Now for a graph $G=(V,E)$ and a subset $W \subset V$,
cut${}_{W} (G)$ denotes the number of edges with precisely one endpoint in
$W$, and maxcut$(G)$ is the maximum
value of cut${}_{W} (G)$ taken over all subsets $W \subset V$.

The next lemma follows immediately from Theorem 1 in~\cite{Buchheim-Zheng}.
We sketch the proof for the sake of completeness.

\begin{lemma}\label{lem:reform}
\[
\Crt(K_{n}) = |E_n| - \mbox{\rm maxcut}(G_n).
\]
\end{lemma}
\proof
Given a two page (circle) drawing of $K_n$, define $W \subset V_n$ as the
chords that are drawn inside the circle.
The edges of $E_n$ with precisely one endpoint in $W$ now 
correspond to edges of $K_n$ that do not cross in the drawing.
\qed


\vglue 0.4 cm

As a consequence of this lemma, one may calculate $\Crt(K_{n})$ for fixed
(in practice, sufficiently small) values of $n$ by solving a maximum cut problem.  This was done
by Buchheim and Zheng \cite{Buchheim-Zheng} for $n \le 13$, by solving
the maximum cut problem with a branch-and-bound algorithm (Bucheim and
Zheng applied the technique to many other graphs as well).  Using the
{\tt BiqMac} solver~\cite{BiqMac}, we have computed the exact value of $\Crt(K_n)$ for
$n\le 18$ and for $n \in \{20,24\}$ (statement (A) in Theorem~\ref{thm:main2}; see Section~\ref{sec:numerical}).


\section{The Goemans-Williamson max-cut bound}
\label{sec:GW}

We follow the standard practice to use
$\mathbb{R}^{p\times q}$
(respectively, $\mathbb{C}^{p\times q}$)
to denote the space of $p\times q$ matrices over $\mathbb{R}$
(respectively, $\mathbb{C}$).
For $\vector{A} \in \mathbb{R}^{p \times p}$, the notation
$\vector{A} \succeq 0$ means that $\vector{A}$ is symmetric positive semidefinite,
whereas for
$\vector{A} \in \mathbb{C}^{p\times p}$, it means that $\vector{A}$ is Hermitian positive semidefinite.

Let $G$ be a graph with $p$ vertices,
and let $\vector{L}$ be
its Laplacian matrix.
Goemans and Williamson \cite{goe95} introduced the following semidefinite programming-based upper bound
on $\mbox{maxcut}(G)$:
\begin{equation}\label{eq:gw}
\mbox{maxcut}(G) \le \mathcal{GW}(G) := \max \left\{
\frac{1}{4} \trace(\vector{L}\vector{X}) \; \biggl| \; \vector{X}
\succeq 0,\ X_{ii} = 1 \; (1 \le i \le p) \right\}.
\end{equation}

It was shown in \cite{goe95} that $0.878\mathcal{GW}(G) \le \mbox{maxcut}(G) \le \mathcal{GW}(G) $ holds for all graphs $G$.


The associated dual semidefinite program takes the form:
\begin{equation}
\mathcal{GW}(G) = \min_{\vector{w} \in \mathbb{R}^{p}} \left\{ \sum_i w_i \; \biggl| \; \Diag(\vector{w}) - \frac{1}{4}\vector{L} \succeq 0\right\},
\label{dual GW}
\end{equation}
where ${\rm Diag}$ is the operator that maps a $p$-vector to a $p\times p$ diagonal
matrix in the obvious way.

\subsection{The Goemans-Williamson bound for $G_n$}
Using the technique of symmetry reduction for semidefinite programming (see e.g.\ \cite{GaPa}), one can simplify the
dual  problem (\ref{dual GW}) for the graphs $G_n$ defined in Section~\ref{sec:maxcutformulation}, by using the dihedral automorphism group of $G_n$.
We state the final expression as the following lemma.

\begin{lemma}
\label{lemma:reformulation SDP}
Let $n>0$ be an odd integer and $d = \lfloor n/2 \rfloor$. One has
\[
\mathcal{GW}(G_n) = \min_{{y} \in \mathbb{R}^{d-1}} \left\{n\sum_{i=2}^d y_i
\; \left| \;
 \mbox{\rm Diag}\left({y} - \frac{1}{4} {val}\right) + \Lambda^{(m)} \succeq 0 \; (0\le m \le d)\right.\right\},
\]
where
\begin{eqnarray}
val_i & = & i(i-1) + 2(i-1)(d-i), \quad 2\le i \le d,  \nonumber \\ 
\Lambda^{(m)}_{ij}  &=&  \frac{1}{4}\left[\sum_{k=1}^{i-1} e^{\frac{-2\pi mk\sqrt{-1}}{n}} +  \sum_{k=n-j+1}^{n-j+i-1} e^{\frac{-2\pi mk\sqrt{-1}}{n}}  \right], \;\;\; {2 \le i\le j \le d}, \label{GWconstraints} \\
\Lambda^{(m)} &= &{\Lambda^{(m)}}^*  \in \mathbb{C}^{d-1 \times d-1}. \nonumber 
\end{eqnarray}
\end{lemma}

For the proof, we recall that the {\em Kronecker product} $\vector{A} \otimes \vector{B}$ of matrices $\vector{A} \in
\mathbb{R}^{p \times q}$ and $\vector{B}\in \mathbb{R}^{r\times s}$ is
defined as the $pr \times qs$ matrix composed of $pq$ blocks of size
$r\times s$, with block $ij$ given by $a_{ij}\vector{B}$ where $1 \le i \le p$ and $1 \le j \le q$.

\proof
We first label the vertices $G_n$ as follows.
Consider the cycle $C_n$ with vertices numbered $\{0,1,\ldots,n-1\}$ in the usual way.
The vertices of $G_n$ that correspond to chords connecting points at cyclic distance $i$ are now
given consecutive labels $(0,i), (1,i+1), \ldots (n-1,i-1)$.
Thus the adjacency matrix of $G_n$ is partitioned into a block structure, where
 each row of blocks is indexed by a cyclic distance $i \in \{2,\ldots,d\}$,
 and each block has size $n\times n$.

Moreover, block $(i,j)$ ($i,j \in \{2,\ldots,d\}, \; i \le j$) is given by the $n \times n$ circulant matrix with first row
\[
[0 \; \mathbf{1}_{i-1}^T \; \mathbf{0}_{n-i-j+1}^T \; \mathbf{1}_{i-1}^T \; \mathbf{0}_{j-i}^T],
\]
where $\mathbf{1}_k$ and $\mathbf{0}_k$ denote the all-ones
and all-zeroes vectors in $\mathbb{R}^k$, respectively.

The eigenvalues of this block are
\begin{equation}
\label{eigs}
\lambda_m = \sum_{k=1}^{i-1} e^{\frac{-2\pi mk\sqrt{-1}}{n}} +  \sum_{k=n-j+1}^{n-j+i-1} e^{\frac{-2\pi mk\sqrt{-1}}{n}} \;\;\;\;\; (0 \le m \le n-1);
\end{equation}
see e.g.\ \cite{circulant matrices}.

Now let an optimal solution $\vector{w}$ of the semidefinite program (\ref{dual GW}) be given for $G = G_n$.
If we project the matrix
\[
\Diag(\vector{w}) + \frac{1}{4}\vector{L}
\]
onto the centralizer ring of $\mbox{Aut$(G_n)$}$, then we again obtain an optimal solution.
Indeed, this projection simply averages the components of $w$  over the $d-1$ orbits of  $\mbox{Aut$(G_n)$}$.
Moreover, the projection is also a symmetric positive semidefinite matrix, since any projection
of a Hermitian  positive semidefinite matrix onto a matrix $*$-algebra is again positive semidefinite (see e.g.\ \cite{towers of algebras}).

Denoting the average of the $w$ components in orbit $i$ by $y_i$,
we obtain an optimal solution of the form $$\mathcal{GW}(G_n) = \min_{\vector{y} \in \mathbb{R}^{d-1}} n\sum_{i=2}^d y_i$$
such that
\begin{equation}
\label{lmi}
\sum_{i=2}^d y_i\left(\vector{e}_{i-1}\vector{e}_{i-1}^T\right)\otimes \vector{I}_n - \frac{1}{4}\vector{L} \succeq 0,
\end{equation}
where $\vector{e}_i$ denotes
the $i$-th standard unit vector in $\mathbb{R}^{d-1}$, and $\vector{I}_n$ denotes the identity matrix of order $n$.

Let $\vector{Q}$ denote the (unitary) discrete Fourier transform matrix of order $n$.
Condition (\ref{lmi}) is equivalent to
\begin{equation}
\label{lmi2}
(\vector{I}_n \otimes \vector{Q}) \left(\sum_{i=2}^d y_i\left(\vector{e}_{i-1}\vector{e}_{i-1}^T\right)\otimes \vector{I}_n - \frac{1}{4}\vector{L}\right)(\vector{I}_n \otimes \vector{Q})^* \succeq 0.
\end{equation}
Since the unitary transform involving $\vector{Q}$ diagonalizes any circulant matrix (see e.g.\ \cite{circulant matrices}),
the matrix
$(\vector{I}_n \otimes \vector{Q})\vector{L}(\vector{I}_n \times \vector{Q})^*$ becomes a block matrix where each $n\times n$ block is diagonal, with diagonal entries of block $(i,j)$ given
by the eigenvalues in (\ref{eigs}).

Finally, the rows and columns of the left hand side of (\ref{lmi2}) may now be re-ordered to form a block diagonal
matrix with $n\times n$ diagonal blocks given by the {right} hand side of (\ref{GWconstraints}) (only $d+1$ of these blocks are distinct).
This completes the proof. \qed

A few remarks on the semidefinite programming reformulation in Lemma \ref{lemma:reformulation SDP}:
\begin{itemize}
\item
The constraints involve Hermitian (complex) linear matrix inequalities, as opposed to the
real symmetric linear matrix inequalities in (\ref{dual GW}).
\item
The reduced problem has $d+1$ linear matrix inequalities involving $(d-1)\times (d-1)$ matrices. By comparison,
the original problem had one linear matrix inequality involving $({n \choose 2}-n)\times ({n \choose 2} - n)$ matrices.
As a result, the reformulation of $\mathcal{GW}(G_n)$ may be solved for much larger values of $n$ than
the original formulation (\ref{dual GW}) (see next section).
\item
{
Although we have only done the symmetry reduction of problem (\ref{dual GW}) for $G_n$ with $n$ odd, the
case for even $n$ is similar, but omitted, since we will not use it later.
\item
Any feasible point $y \in \mathbb{R}^{d-1}$ of the reduced problem in Lemma \ref{lemma:reformulation SDP}
provides a certificate of an upper bound on $\mathcal{GW}(G_n)$,
and consequently a certificate of a lower bound on $\Crt(K_n)$, since $\Crt(K_n) \ge {n\choose 4} - \mathcal{GW}(G_n)$.
}
\end{itemize}

\section{Numerical computations: proof of Theorem~\ref{thm:main2}}\label{sec:numerical}

Theorem~\ref{thm:main2}~(A)
follows by an exact computation of the related maxcut problem of $G_n$
for certain values of $n$, while Theorem~\ref{thm:main2} (B) follows
by a calculation of $\mathcal{GW}(G_{899})$ and a standard counting argument.

\subsection{Proof of (A)}

First we observe that if $n < 5$ then $Z(n)=0$, and the assertion $\Crt(K_n) = Z(n)$ is
easily verified.

We computed the exact value $\mbox{maxcut}(G_{n})$ for $n=5, 7, 9, 11, 13, 15, 17,
20$,  and $24$,
using the solver {\tt BiqMac} \cite{BiqMac},
available from \url{http://biqmac.uni-klu.ac.at/}.
Computation was done on a quad-core 2.0 GHz Intel PC with 10 GB of RAM
memory, running Linux.
We used a cut-off time of $60$ hours for the computation for each value of $n$. As a consequence, the
{\tt BiqMac} solver failed to terminate successfully in a few cases, namely $n =19,21,22$, and $23$.

 The results are presented in the second column
of Table~\ref{tab:table1}. The exact value of $\Crt(K_n)$ (fourth column) follows from the second
and third columns (using Lemma~\ref{lem:reform}). The fifth column is
given for reference, to verify that $\Crt(K_n) = Z(n)$ for all these
values of $n$.  Thus (A) follows for $n=5,7,9,11,13,15,17,20,$ and
$24$.
{
The last two columns show the CPU time required, and the number of nodes
evaluated in the branch an bound tree by the solver  {\tt BiqMac}.
}
Finally, an elementary, well-known counting argument shows that if
$\Crt(K_{2m+1}) = Z(2m+1)$ for some positive integer $m$, then
$\Crt(K_{2m+2}) = Z(2m+2)$. This proves (A) for the remaining cases
$n=6,8,10,12,14,16$, and $18$.

\renewcommand{\arraystretch}{1.5}
\begin{table}
\begin{center}
\begin{tabular}{|c|c|c|c|c|c|c|}\hline \label{tab:table1}
$n$ & $\mbox{maxcut}(G_{n})$ & $|E_n|= \binom{n}{4}$ & $\Crt(K_n)$ &
  $Z(n) $ & CPU time (s) & Branch \& bound nodes\\ \hline
$5$ & 4 & 5 & 1 & 1 & 0.001 & 1\\ \hline
$7$ & 26 & 35 & 9 & 9 & 0.01 & 1 \\ \hline
$9$ & 90 & 126 & 36 & 36 & 0.22 & 3\\ \hline
$11$ & 230 & 330 & 100 & 100  & 4.01 & 17 \\ \hline
$13$ & 490 & 715 & 225 & 225 & 73.27 & 151\\ \hline
$15$ & 924 & 1,365 & 441 & 441 & 906.61 & 841 \\ \hline
$17$ & 1,596 & 2,380 & 784 & 784 & 15,542 & 6,837 \\ \hline
$20$ & 3,225 & 4,845 & 1,620 & 1,620 & 58,784 & 9,479\\ \hline
$24$ & 6,996 & 10,626 & 3,630 & 3,630 & 5,616 & 65\\ \hline
\end{tabular}
\caption{The second column gives the exact values of
  $\mbox{maxcut}(G_n)$ that we computed. The fourth column gives the
  corresponding exact values of $\Crt(K_n)$ (using that $\Crt(K_{n}) = |E_n|
  - \mbox{\rm maxcut}(G_n)$). For all these values of $n$,  the conjecture
  $\Crt(K_n) = Z(n)$ is verified.
 }
\end{center}
\end{table}

\subsection{Proof of (B)}


The first ingredient in the proof of (B) is
 a lower bound for $\Crt(K_{899})$. We obtained this bound via the approximate calculation of
 $\mathcal{GW}(G_{899})$, which we achieved by using the semidefinite programming reformulation
 in Lemma~\ref{lemma:reformulation SDP}.
 Computation was done on a Dell Precision T7500 workstation with 92GB of RAM memory, using the semidefinite
 programming solver SDPT3~\cite{SDPT3-ref1,SDPT3-ref2} under Matlab 7 together with the Matlab package YALMIP~\cite{YALMIP}.
 The total running time was $12,602$ seconds.
 SDPT3 was chosen since it  can deal
 with Hermitian matrix variables.
 We obtained $\mathcal{GW}(G_{899})
\le 1.76537474 \times 10^{10}$.
Using
 Lemma~\ref{lem:reform} and (\ref{eq:gw}), it follows immediately that
 \begin{equation}\label{eq:lowbo}
 \Crt(K_{899}) \ge 9,381,181,976.
 \end{equation}

The second ingredient to prove (B) is to establish a lower bound on the
asymptotic ratio $\lim_{n\to\infty}\Crt(K_n)/{Z(n)}$ that can be
guaranteed from a lower bound on $\Crt(K_m)$ for some $m > 3$.

\begin{claim}\label{cla:claimA}
For any  integer $m > 3$,
\[
\lim_{n\to\infty}
\frac{\Crt(K_{n})}{Z(n)} \ge
\frac{64}{m(m-1)(m-2)(m-3)}\ \Crt(K_{m}).
\]
\end{claim}

\begin{proof}
Let $m,n$ be  integers with $3 < m < n$.
Consider a $2$-page drawing $D$ of $K_n$ with $\Crt(K_n)$ edge crossings.
Let $\mathcal{G}$ denote the set of  subgraphs of $K_n$ that are isomorphic to $K_m$, i.e.\ $|\mathcal{G}| = {n \choose m}$.
Any two disjoint edges in $K_n$ occur in
${n-4 \choose m-4}$ of the graphs in $\mathcal{G}$.
Thus, every crossing in $D$ appears in the induced drawings of
${n-4 \choose m-4}$ graphs in $\mathcal{G}$. Consequently,
\begin{equation*}
\Crt({K_n})  \ge  \frac{\Crt(K_m){n \choose m}}{{n-4 \choose m-4}}
          = \frac{\Crt(K_{m})n(n-1)(n-2)(n-3)}{m(m-1)(m-2)(m-3)}.
\end{equation*}

The claim follows immediately from this inequality
and the definition of $Z(n)$.
\end{proof}

It only remains to observe that (B) is an immediate
consequence of (\ref{eq:lowbo}) and Claim~\ref{cla:claimA}.

\section{A quadratic programming lower bound for $\Crt({K_{m,n}})$}\label{sec:kmnone}

Throughout this section,
assume that $m$ is fixed, and
consider $2$-page drawings of $K_{m,n}$, where $n$ is any positive
integer. Thus, all vertices lie
on the $x$-axis, and each edge is contained either in the upper or
in the lower half-plane. We assume, without any loss of generality,
that the $m$ degree-$n$ {\em blue} vertices $b_1, b_2, \ldots, b_m$
appear on the $x$-axis in this order, from left to right.  The $n$ degree-$m$
vertices are {\em red}. The {\em star} of a red vertex $r$ (which we
shall denote $\star{r}$)
is the subgraph induced by $r$ and its incident edges. Thus, for every
red vertex $r$, $\star{r}$ is isomorphic to $K_{m,1}$.

\subsection{The {type} of a red vertex}

In our quest for lower bounding the number of crossings in any
$2$-page drawing $\dd$ of $K_{m,n}$, the strategy is to consider any two red vertices $r,r'$, and find a
lower bound for the number $\numcr_{{}_\dd}(\star{r},\star{r'})$ of
crossings in $\dd$ that
involve one edge in $\star{r}$ and one edge in $\star{r'}$. The bound
we establish is in terms of the {\em types} of $r$ and
$r'$. The type (formally defined shortly) of a red vertex is determined by
its position relative to the blue vertices, and by which edges
incident with it lie on each half-plane.

We start by noting that we may focus our interest in drawings in which no red vertex lies
to the left of $b_1$. Indeed, if the leftmost red vertex
lies to the left of $b_1$ (and so it is the leftmost vertex overall),
it is easy to see that it may be moved so
that it becomes the rightmost (overall) vertex, without increasing the
number of crossings. By repeating this procedure we get a drawing with
the same number of crossings, and with no red vertex to the left of
$b_1$. Thus there is no loss of generality in dealing only  with drawings that satisfy this property, and
it follows that each red vertex $r$ has a {\em
  position} $p(r)$ relative to the blue points: $p(r)$ is the largest
$j\in\{1,2,\ldots,m\}$ such that $r$ is to the right of $b_j$.

Also, to each red vertex $r$ we can naturally assign a partition
$\{U(r),L(r)\}$ of $\{1,2,\ldots,m\}$, the {\em distribution} of $r$,
defined by the rule that $j\in\{1,2,\ldots,m\}$ is in $U(r)$
(respectively, $L(r)$) if the
edge $rb_j$ lies in the upper (respectively, lower)
half-plane.
We call the triple $(p(r),
{U}(r),L(r))$ the {\em type} of $r$, and denote it by $\type{r}$.
Since $p(r)$ can be any integer in $\{1,2,\ldots,m\}$, and $U(r)$ any
subset of ${\{1,2,\ldots,m\}}$ (and $L(r)=\{1,2,\ldots,m\}\setminus
U(r)$ is determined by $U(r)$),
it follows that there are $m2^m$ possible types for a red vertex.  We
use Types$(m)$ to denote the collection of all $m2^m$ possible
types.

\subsection{Guaranteeing crossings between red stars using types}

The motivation for introducing the concept of type is that knowing the types of
two red vertices $r$ and $r'$ in a drawing $\dd$ of $K_{m,n}$ yields a lower bound on
$\numcr_{{}_\dd}(\star{r},\star{r'})$.
We illustrate this with an example. Suppose that $m=5$, and that
$\type{r} = (2,\{1,2,3,5\},\{4\})$ and $\type{r'} = (3,\{1,3,4,5\},\{2\})$. The
situation is thus as illustrated in Figure~\ref{fig:types}.

\begin{figure}[ht]
\centering
\resizebox{12cm}{!}{\input{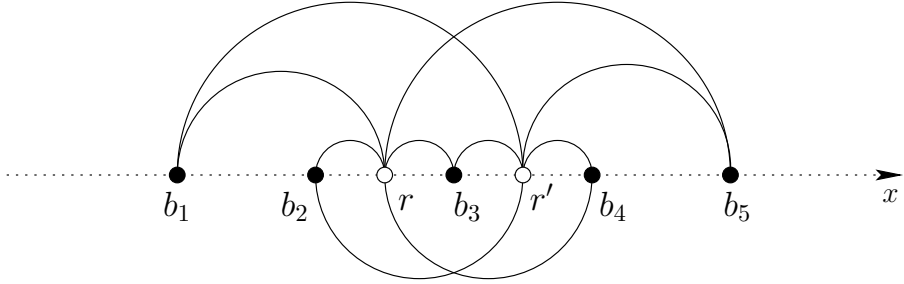}}
\caption{The types of the red vertices $r$ and $r'$ are
$(2,\{1,2,3,5\},\{4\})$ and $\type{r'} = (3,\{1,3,4,5\},\{2\})$, respectively. Thus,
  $r$ is in position $2$ (that is, between $b_2$ and
  $b_3$), and  the edges joining $r$ to $b_1, b_2, b_3$ and $b_5$ are
  in the upper half-plane and the edge joining $r$ to $b_4$ is in the
  lower half-plane. Both
  crossings in this drawing can be easily predicted from
  $\type{r}$ and $\type{r'}$.}
\label{fig:types}
\end{figure}

Both crossings between $\star{r}$ and $\star{r'}$ in this example are
easily detected from $\type{r}$ and $\type{r'}$. Indeed, since $b_1,
r, r', b_5$ occur in this order from left to right (this follows since
$r$ and $r'$ are in positions $2$ and $3$, respectively), and $b_1 r'$
and $r b_5$ are both on the upper half-plane (this follows since
$1\in U(r')$ and $5\in U(r)$), it follows that $b_1 r'$
and $r b_5$ must cross.  We remark that the key pieces of information
are that (i) the endpoints $b_1, r, r', b_5$ of $b_1 r'$ and $r b_5$ {\em alternate} on
the $x$-axis (that is, they are all distinct and occur in the $x$-axis
so that the ends of one edge are in first and third place and the ends
of the other edge are in second and fourth place); and (ii) both edges are drawn on the same half-plane.

Using this simple criterion (if two edges are on the same half-plane
and their endpoints alternate, then they must cross each other), given two red points $r,r'$ in a
drawing $\dd$ of $K_{m,n}$, it is easy to derive a lower bound for
$\numcr_{{}_\dd}(\star{r},\star{r'})$ in terms of $\type{r}$ and
$\type{r'}$.  This bound (Proposition~\ref{pro:startype} below) is
given in terms of a quantity we now proceed to define.

First, for $\sigma=(p,U,L)$ and $\tau=(p',U',L')\in \hbox{\rm Types}(m)$, we let
\begin{align*}\label{eq:forpos1}
[\sigma,\tau] :=
&\biggl|
\biggl\{
(i,j) \ \bigl| \ \biggl(
\bigl(i \in U\ \hbox{\rm and } j\in U'\bigr)\ \hbox{\rm or }
\bigl(i \in L\ \hbox{\rm and } j\in L'\bigr)
\biggr)\
\hbox{\rm and }\nonumber \\
&\biggl(
\bigl(
i < j \le p
\bigl)
\
\hbox{\rm or }
\bigl(
j \le p \ \text{\rm and }
p' < i
\bigl)
\
\hbox{\rm or }
\bigl(
i < j \ \text{\rm and }
p' < i
\bigl)
\
\hbox{\rm or }
\bigl(
p < j < i \le p'
\bigl)
\
\biggr)
\biggr\}
\biggr|,
\nonumber
 \end{align*}
and
\begin{equation*}
{Q_{\sigma\tau}} :=
\begin{cases}
[\sigma,\tau], & \text{if $p < p'$,} \\
[\tau,\sigma], & \text{if $p > p'$,} \\
\min\bigl\{[\sigma,\tau],[\tau,\sigma] \bigr\}, \hbox{\hglue 0.5 cm}
& \text{if $p = p'$.}
\end{cases}
\end{equation*}

The nonnegative integers $Q_{\sigma\tau}$ can be naturally regarded as the entries
of a $m2^m \times m2^m$-matrix $\vector{Q}$ indexed (both by rows and
columns) by the elements of Types$(m)$.
It is easy to check that the matrix $\vector{Q}$ is symmetric, and its entries provide the
lower bounds we have been aiming for.

\begin{proposition}\label{pro:startype}
Let $\sigma,\tau \in \hbox{\rm Types}(m)$, and
suppose that $r_\sigma, r_\tau$ are red points in a drawing $\dd$ of
$K_{m,n}$, such that
$\type{r_\sigma} = \sigma$ and $\type{r_\tau} = \tau$.
Then
\begin{equation*}
\numcr_{{}_\dd}(\star{r_\sigma},\star{r_\tau}) \ge Q_{\sigma\tau}.
\end{equation*}
\end{proposition}

\begin{proof}
Suppose first that $r_\sigma$ occurs to the left of $r_\tau$. It is easy to verify that if $i,j$ are integers such that
either (i) $i<j\le p$; (ii) $j \le p$ and $p' < i$; or (iii) $i<j$ and
$p'<i$; or (iv) $p<j<i \le p'$, then the endpoints of $rb_i$ and
$r'b_j$ alternate. Therefore, if either $i\in U$ and $j\in U'$, or
$i\in L$ and $j\in L'$, then $rb_i$ and $r'b_j$ cross each
other. Therefore there is an injection from the set of all pairs $(i,j)$ of integers that satisfy the
condition in the definition of $[\sigma,\tau]$, to the set of
crossings that involve an edge in $\star{r_\sigma}$ and an edge in
$\star{r_\tau}$; that is,
$\numcr_{{}_\dd}(\star{r_\sigma},\star{r_\tau}) \ge [\sigma,\tau]$.

Similarly, if $r_\sigma$ occurs to the right of $r_\tau$, then
$\numcr_{{}_\dd}(\star{r_\sigma},\star{r_\tau}) \ge [\tau,\sigma]$.

Now if $p < p'$ (respectively, $p > p'$), then $r_\sigma$ necessarily occurs to
the left (respectively, to the right) of $r_\tau$, and so it follows
that $\numcr_{{}_\dd}(\star{r_\sigma},\star{r_\tau}) \ge [\sigma,\tau]
= Q_{\sigma\tau}$ (respectively, $\ge [\tau,\sigma] =
Q_{\sigma\tau}$), as required.
Finally, If $p=p'$, then $r_\sigma$ can be either to the right
or to the left of $r_\tau$.  In the first case,
$\numcr_{{}_\dd}(\star{r_\sigma},\star{r_\tau}) \ge [\sigma,\tau]$,
while in the second case
$\numcr_{{}_\dd}(\star{r_\sigma},\star{r_\tau}) \ge
[\tau,\sigma]$. Thus, in this case,
$\numcr_{{}_\dd}(\star{r_\sigma},\star{r_\tau}) \ge
\min\{[\sigma,\tau],[\tau,\sigma]\}=Q_{\sigma\tau}$, as required.
\end{proof}

\subsection{The quadratic program}

Consider now any fixed $2$-page drawing $\dd$ of $K_{m,n}$. For each type
$\sigma\in \hbox{\rm Types}(m)$, let $n_\sigma$ denote the number of red
vertices whose type in $\dd$ is $\sigma$,  let $p_\sigma:=n_\sigma/n$,
and let $\vector{p}$ be the vector
$(p_\sigma)_{\sigma\in\hbox{\rm\footnotesize Types}(m)}$. It follows
  immediately from Proposition~\ref{pro:startype} that the number
  $\Crt(\dd)$ of crossings in $\dd$ satisfies
\begin{align*}
\Crt(\dd) &\ge \frac{1}{2} \sum_{\stackrel{\sigma,\tau\in\hbox{\rm\footnotesize
    Types}(m)}{\sigma\neq \tau}} Q_{\sigma\tau} n_\sigma n_\tau
+ \sum_{\sigma\in\smtyp} Q_{\sigma\sigma} {n_\sigma\choose 2}\\
&= \frac{1}{2} \sum_{\sigma,\tau\in\smtyp} Q_{\sigma\tau} n_\sigma n_\tau - \frac{1}{2} \sum_{\sigma\in\smtyp} Q_{\sigma\sigma} n_\sigma\\
&= \frac{n^2}{2} \vector{p}^T \vector{Q} \vector{p} - \frac{n}{2} \sum_{\sigma\in\smtyp} Q_{\sigma\sigma} p_\sigma\\
&\ge \frac{n^2}{2} \vector{p}^T \vector{Q} \vector{p} - \frac{n}{2} \max_{\sigma\in\smtyp} Q_{\sigma\sigma}\\
&\ge \frac{n^2}{2} \vector{p}^T \vector{Q} \vector{p} - \frac{m(m-1) n}{4},
\end{align*}

\noindent where for the last inequality we use that
$\Sigma_{\sigma\in\hbox{\rm\tiny Types}(m)} p_\sigma = 1$ and that $Q_{\sigma\sigma}=[\sigma,\sigma] \le {m \choose 2}$.

%
%
%

The derived inequality holds for every $2$-page drawing $\dd$ of $K_{m,n}$,
and so in particular for a crossing-minimal drawing. Thus, if we let
\[
\Delta = \left\{\vector{x} = (x_1, x_2, \ldots, x_{m2^m})^T \in \mathbb{R}^{m2^m} \; \biggl| \; \sum_i x_i = 1, \; x_i \ge 0 \right\}
\]
denote the standard simplex in $\mathbb{R}^{m2^m}$, then we obtain

\begin{equation}\label{eq:genbound}
\Crt({K_{m,n}}) \ge \frac{n^2}{2} \biggl(\min_{\vector{x}\in \Delta }
\, \vector{x}^T \vector{Q} \vector{x}\biggr) - \frac{m(m-1) n}{4}.
\end{equation}

We may therefore obtain a lower bound on $\Crt({K_{m,n}})$ for some
fixed $m$ (we will be particularly interested in the case $m = 7$), by solving the
standard quadratic programming problem
\begin{equation}
\label{eq:QP}
\lb(m) = \min_{\vector{x} \in \Delta} \,\vector{x}^T\vector{Q}\vector{x}.
\end{equation}

The standard quadratic programming problem is NP-hard in general, and we will only compute a lower bound on the
minimum via semidefinite programming, as explained in the next section.

\section{A semidefinite programming lower bound on
  $\Crt({K_{m,n}})$}\label{sec:kmntwo}

The usual semidefinite programming relaxation of problem~(\ref{eq:QP}) takes the form
\begin{eqnarray}
\lb(m) &\ge & \min \bigl\{\trace(\vector{Q}\vector{X}) \; \bigl|
\; \trace(\vector{J}\vector{X}) = 1, \; \vector{X} \succeq \vector{0}, \;
\vector{X} \ge \vector{0}\bigr\} \nonumber \\
& = & \max \bigl\{ t \; \bigl| \;
\vector{Q} - t\vector{J} = \vector{S_1} + \vector{S_2}, \;
\vector{S_1} \succeq \vector{0}, \; \vector{S_2} \ge \vector{0}\bigr\}
\nonumber \\
&:= & {\text{\rm SDP}}_{bound}(m), \label{eq:spb}
\end{eqnarray}
where $\vector{J}$ is the all-ones matrix,  and $\vector{X} \ge \vector{0}$
means that $\vector{X}$ is entrywise nonnegative. We observe that the first equality is due to the duality theory of semidefinite programming.

Due to the special structure of $\vector{Q}$, we may again use symmetry reduction to
reduce the size of these problems.
To this end, for odd $m$, we may order the rows and columns of $\vector{Q}$ to obtain a block matrix consisting
of circulant blocks of order $2m$. (Thus there are $2^{m-1}$ rows/columns of blocks).
{
The ordering of rows works as follows: we first define a group action on the set Types$(m)$.
For ease of notation we now represent the elements of
Types$(m)$ as $(p,U)$, with $p \in \{0,\ldots,m-1\}$ and $U \subseteq \{0,\ldots,m-1\}$, i.e.\ we now number the $m$ vertices
from $0$ to $m-1$, and omit the set $L$ (which is redundant in the description since it is the complement of $U$).

The group in question is generated by the following two elements, a 'flip':
\[
g_1:(p,U) \mapsto (p,\{0,\ldots,m-1\}\setminus U),
\]
and a 'cyclic shift':
\[
g_2:(p,U) \mapsto (p+1 \mod m,\{u+1 \mod m \; | \; u \in U \}).
\]
Note that $g_1$ and $g_2$ commute and therefore generate an Abelian group of order $2m$.
If $m$ is odd, then $g := g_1 \circ g_2$ generates the entire group, i.e.\ in this case
we obtain the cyclic group of order $2m$. Indeed, the order of $g$ equals the least common multiple
of the orders of $g_1$ and $g_2$, namely $2m$ if $m$ is odd.

Also note that
\[
Q_{\sigma, \tau} = Q_{g_i(\sigma), g_i(\tau)} \quad \forall \sigma, \tau \in \mbox{Types}(m), \; i \in \{1,2\},
\]
i.e.\ the crossing number of a 2-page drawing does not change if we 'flip' the drawing along its spine, or, in the circular model, rotate the drawing.

Finally, we group together the $2m$ elements of Types$(m)$ that belong to a given orbit of the group, to
obtain $2m \times 2m$ circulant blocks.
}
In what follows, we denote the first row of the $2m \times 2m$ circulant block
$(i,j)$ by $q^{(i,j)} \in \mathbb{Z}^{2m}$.

\begin{lemma}
\label{lemma:reformulation SDP2}
{For odd $m$}, the semidefinite programming bound (\ref{eq:spb}) may be reformulated as
\[
\text{\rm SDP}_{bound}(m) = \max t
\]
subject to
{
\begin{eqnarray*}
q^{(i,j)}_k - t - x^{(i,j)}_k &\ge& 0, \quad  0\le k \le 2m-1, \; 1 \le i,j \le 2^{m-1}, \\
X^{(t)}_{ij} &=&  x^{(i,j)}_0 + \sum_{k=1}^{2m-1} x^{(i,j)}_k e^{-\pi \sqrt{-1} tk/m}, \;\;\; 1 \le i \le j \le 2^{m-1}, \; 0 \le t \le 2m-1,\\
X^{(t)} = ( X^{(t)})^*&\succeq& \vector{0}, \quad 0 \le t \le 2m-1,\\
x^{(i,i)}_k -x^{(i,i)}_{2m+1-k} &=&  0, \quad 1 \le k \le m-1, \; 1\le i \le 2^{m-1}, \\
x^{(i,j)} & \in & \mathbb{R}^{2m}, \quad 1 \le i,j \le 2^{m-1}.
\end{eqnarray*}
}
\end{lemma}
\proof
The proof is similar to that of Lemma \ref{lemma:reformulation SDP} and is therefore omitted.\qed

A few remarks on the semidefinite programming reformulation in Lemma \ref{lemma:reformulation SDP2}:
\begin{itemize}
\item
As in  Lemma \ref{lemma:reformulation SDP}, the constraints involve Hermitian (complex) linear matrix inequalities.
\item
The reduced problem has $2m$ linear matrix inequalities involving $(2^{m-1})\times (2^{m-1})$ matrices. By comparison,
the original problem had one linear matrix inequality involving a $(m2^m)\times (m2^m)$ nonnegative matrix.
As a result, the reformulation in Lemma \ref{lemma:reformulation SDP2} may be solved for larger values of $m$ than
the original formulation (see next section).
\item
{
Similarly to  Lemma \ref{lemma:reformulation SDP}, every feasible point $x^{(i,j)}\in  \mathbb{R}^{2m}$ $( 1 \le i,j \le 2^{m-1})$
yields a certificate of lower bound on $\text{\rm SDP}_{bound}(m)$, and consequently a certificate of a lower
bound on $\Crt({K_{m,n}})$, by (\ref{eq:genbound}).
}
\end{itemize}

\section{Numerical computations: proof of Theorem~\ref{thm:main1}}\label{sec:numericalzar}


Using the reformulation in Lemma~\ref{lemma:reformulation SDP2}, we showed
numerically that $\text{\rm SDP}_{bound}(7) =
\frac{9}{2}$.
 Computation was done on a Dell Precision T7500 workstation with 92GB of RAM memory, using the semidefinite
 programming solver SDPT3~\cite{SDPT3-ref1,SDPT3-ref2} under Matlab 7 together with the Matlab package YALMIP~\cite{YALMIP}.
 The running time was $23,774$ seconds.
 SDPT3 was chosen since it  can deal
 with Hermitian matrix variables.

Using that $\text{\rm SDP}_{bound}(7) =9/2$,
it follows from (\ref{eq:genbound}), (\ref{eq:QP}), and (\ref{eq:spb})
that
\begin{equation}\label{eq:k7n}
\Crt(K_{7,n}) \ge ({9}/{4})n^2 - {(21/2)n}.
\end{equation}

We recall that $Z(7,n) = 9\lfloor{n/2}\rfloor
\lfloor{(n-1)/2}\rfloor = (9/4)n^2 + O(n)$, and that $\Crt(K_{7,n}) \le Z(7,n)$ (since there are $2$-page drawings of
$K_{7,n}$ with exactly $Z(7,n)$ crossings).  Using these observations
and (\ref{eq:k7n}), Theorem~\ref{thm:main1} follows for $m=7$.

Now an elementary counting argument shows
that
$\Crt(K_{8,n}) \ge 8\Crt(K_{7,n})/ 6$, and so using (\ref{eq:k7n}) and simplifying we obtain
$\Crt(K_{8,n}) \ge 3n^2 - 14n$.
Since $Z(8,n)=3n^2 + O(n)$, Theorem 1 follows for $m=8$.

\begin{section}{Concluding remarks}\label{sec:concludingremarks}




The Goemans-Williamson bound (Section~\ref{sec:GW}) empirically  yields
better lower bounds on $\Crt(K_n)$ as $n$ grows; see Figure
\ref{fig:dimret}.

\begin{figure}[h!]
\label{fig:dimret}
\begin{center}
\resizebox{10cm}{!}{\input{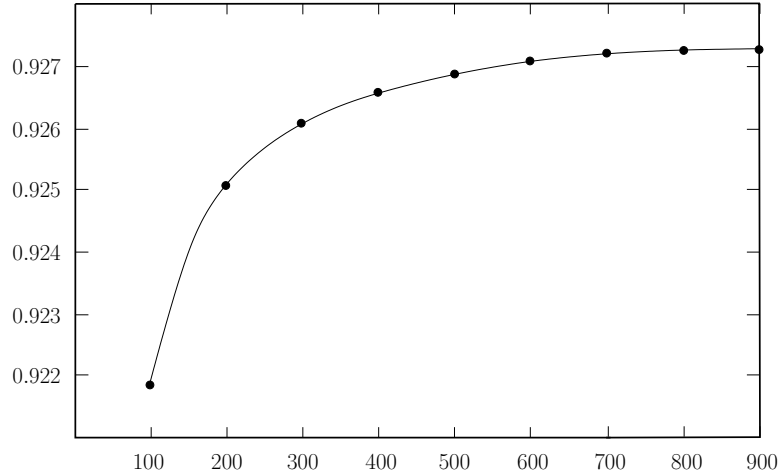}}
\end{center}
\caption{The ratio $\frac{{n \choose 4} - \mathcal{GW}(G_n)}{Z(n)}$
  for $n=99, 199, 299, 399, 499, 599, 699, 799$, and $899$.}
\end{figure}

Based on this empirical evidence, it seems reasonable to expect that
the constant $0.9253$ would be improved if $\mathcal{GW}(G_m)$ were
computed for larger values of $m$.
{
Having said that, the figure also shows a trend of diminishing returns --- by extrapolating the curve in the figure,
it seems that it may not be possible to improve the constant to more than $0.929$, say, through computation
of $\mathcal{GW}(G_m)$, if $m \le 2,000$.

{
Another possibility to improve the constant is to compute $\Crt(K_{m})$ for larger values of $m$ than $m =24$,
by solving the maximum cut problem in Lemma \ref{lem:reform}. If, for example, one could verify in this way that
$\Crt(K_{30}) = Z(30)$, then this would yield the constant $0.9297$, by Claim \ref{cla:claimA}.
}

Regarding the computational lower bound on $\Crt(K_{m,n})$:
It is interesting to note that the SDP bound $\text{\rm SDP}_{bound}(m)$ provided a tight asymptotic bound
on $\Crt(K_{m,n})$ for $m = 3,5$ and $7$.
A similar SDP bound used in \cite{dKMahPasRicSal} and \cite{DeKPasSch} did not provide a tight asymptotic bound on the usual crossing number
$\Cr(K_{m,n})$, not even for $m = 5$.
Our results therefore suggest that one may be able to prove computationally that
$\lim_{n\to\infty} \frac{\Crt(K_{m,n})}{Z(m,n)} = 1$ for (fixed) odd values of $m \ge 9$.
Having said that, for $m=9$, the resulting semidefinite program was too large for us to compute $\text{\rm SDP}_{bound}(9)$.
This problem therefore provides a good future challenge to the computational SDP community.

}

\end{section}

\paragraph{Acknowledgements.}
The authors are grateful to Gelasio Salazar for suggesting to work on these problems, and for 
providing many useful comments, suggestions, and references before deciding to withdraw from this project.
The authors would also like to thank Angelika Wiegele for making the source
code of her max-cut solver {\tt BiqMac} available to them, and Imrich Vrt'o for helpful comments.

 \end{document}